\documentclass[11pt,dvips,twoside,letterpaper]{article}
\usepackage{pslatex}
\usepackage{fancyhdr}
\usepackage{graphicx}
\usepackage{geometry}
 \RequirePackage[T1]{fontenc}

\def\figurename{Figure} 
\makeatletter
\renewcommand{\fnum@figure}[1]{\figurename~\thefigure.}
\makeatother

\def\tablename{Table} 
\makeatletter
\renewcommand{\fnum@table}[1]{\tablename~\thetable.}
\makeatother

\usepackage{amsmath}
\usepackage{amssymb}
\usepackage{amsfonts}
\usepackage{amsthm,amscd}

\newtheorem{theorem}{Theorem}[section]

\newtheorem{proposition}[theorem]{Proposition}
\theoremstyle{definition}
\newtheorem{definition}[theorem]{Definition}

\theoremstyle{remark}
\newtheorem{remark}[theorem]{Remark}

\numberwithin{equation}{section}

\def\P{\mathbb P}

\def\R{\mathbb R}
\def\E{\mathbb E}

\def\E{\mathbb E}

\begin{document}
\title{ {Stochastic viscosity solution for stochastic PDIEs with nonlinear Neumann boundary condition}}
\author{Auguste Aman\thanks{This author is supported by TWAS Research Grants to individuals (No. 09-100 RG/MATHS/AF/\mbox{AC-I}--UNESCO FR: 3240230311),\;\; augusteaman5@yahoo.fr, corresponding author}\\
{\it U.F.R Mathématiques et informatique, Universit\'{e} de Cocody},\\ {\it 582 Abidjan 22, C\^{o}te d'Ivoire}
\\ Yong Ren\thanks{This author is supported by the National Natural Science Foundation of China (No. 10901003) and the Great Research Project of Natural Science Foundation of
Anhui Provincial Universities (No. KJ2010ZD02),\ brightry@hotmail.com and renyong@126.com}\\
{\it Department of Mathematics, Anhui Normal University, Wuhu 241000, China}}

\date{}
\maketitle

\begin{abstract}
This paper is an attempt to extend the notion of viscosity solution
to nonlinear stochastic partial differential integral equations with
nonlinear Neumann boundary condition. Using the recently developed
theory on generalized backward doubly stochastic differential
equations driven by a L\'{e}vy process, we prove the existence of
the stochastic viscosity solution, and further extend the nonlinear
Feynman-Kac formula.
\end{abstract}


\vspace{.08in} \noindent \textbf{Keywords}: Stochastic viscosity
solution, backward doubly stochastic differential equation, L\'{e}vy
process, stochastic partial differential integral equation with
Neumann boundary condition.

\section{Introduction}
The notion of the viscosity solution for partial differential
equations, first introduced by Crandall and Lions \cite{CL}, has an
impact on the modern theoretical and applied mathematics. Today the
theory has become an indispensable tool in many applied fields,
especially in optimal control theory and numerous subjects related
to it. We refer to the well-known "User's Guide" by Crandall et al.
\cite{Cral} and the books by Bardi et al. \cite{Barl} and Fleming
and Soner \cite{FS} for a detailed account for the theory of
(deterministic) viscosity solutions.

Since it is well known that almost all the deterministic problems in
these applied fields have their stochastic counterparts, many works
have extended the notion of viscosity solution to stochastic partial
differential equations (SPDEs, in short). The first among them is
done by Lions and Souganidis \cite{LS1,LS2}. They use the so-called
"stochastic characteristic" to remove the SPDEs. Next, two other
ways of defining a stochastic viscosity solution of SPDEs is
considered by Buckdahn and Ma respectively in \cite{BM1, BM2} and
\cite{BM3}. In the two first paper, they used the "Doss-Sussman"
transformation to connect the stochastic viscosity solution of SPDEs
with the solution of associated backward doubly stochastic
differential equations (BDSDEs, in short). In the second one, they
introduced the stochastic viscosity solution by using the notion of
stochastic sub and super jets. Recently, based on both previous work, Boufoussi et al. introduced
in \cite{Bal}, the notion of viscosity solution of SPDEs with
nonlinear Neumann boundary condition. The existence result is
derived via the so-called generalized BDSDEs and the "Doss-Sussman"
transformation.

Inspired by the aforementioned works, especially \cite{Bal} and
\cite{BM1,BM2}, this paper considers the following nonlinear
stochastic partial differential integral equations (SPDIEs, in
short) with nonlinear Neumann boundary condition
\begin{eqnarray}
\left\{
\begin{array}{ll}
\frac{\partial u}{\partial
t}(t,x)+Lu(t,x)+f(t,x,u(t,x),(u_k^1(t,x))_{k=1}^{\infty})+g(t,x,u(t,x))\dot{B}_{s}=0,
\;\; (t,x)\in[0,T]\times\overline{\Theta},\\\\
\frac{\partial u}{\partial n}(t,x)+\phi(t,x, u(t,x))=0, \;\;\; (t,x)\in[0,T]\times\partial\Theta,\\\\
u(T,x)=u_0(x),\;\;x\in\overline{\Theta},
\end{array}\right.
\label{i1}
\end{eqnarray}
where $\dot{B}$ denotes white noise with respect to Brownian motion $B$; which  Moreover $f,\, g,\, \phi$
and $u_0$ are some measurable functions with appropriate dimensions
and $L$ is the second-order differential integral operator of the
form:
\begin{eqnarray}
L\varphi(t,x)&=&m_1\sigma(x)\frac{\partial\varphi}{\partial x}(t,x)
+\frac{1}{2}\sigma(x)^2\frac{\partial^2\varphi}{\partial x^2}
(t,x)\nonumber\\&&+\int_{\R}\left[\varphi(t,x+\sigma(x)y)-\varphi(t,x)-\frac{\partial\varphi}{\partial
x}
(t,x)\sigma(x)y\right]\nu(dy);\nonumber\\
\label{i2}
\end{eqnarray}
in which $\sigma$ is a certain function and $m_1=\E(L_1)$, which
will be given in Section 3. We denote
\begin{eqnarray*}
\varphi^{1}_k(t,x)=\int_{\R}[\varphi(t,x+\sigma(x)y)-\varphi(t,x)]p_k(y)\nu(dy),\;
k\geq 1
\end{eqnarray*}
and
\begin{eqnarray*}
\frac{\partial\varphi}{\partial n}(t,x)=
\sum_{i=1}^{d}\frac{\partial\psi}{\partial_i}(x)\frac{\partial\varphi}{\partial x_i}(t,x),\; \forall\;
x\in\partial\Theta,
\end{eqnarray*}
where the function $\psi\in C^{2}_{b}(\R^n)$ is connected to the
domain $\Theta$ by the following relation:
\begin{eqnarray*}\Theta=\{x\in\R^{n}:\, \psi(x)>0\} \;\, \mbox{and}\;\;
\partial\Theta=\{x\in\R^{n}:\, \psi(x)=0\}.
\end{eqnarray*}

The goal of this paper is to determine the definition and next
naturally establish
 the existence of the stochastic viscosity solution to SPDIEs \eqref{i1}, which could be used for the purpose of option pricing in a L\'{e}vy market. More precisely,
 we give some direct links between this stochastic viscosity solution and the solution of
 the so-called generalized backward doubly stochastic differential equations  driven by
 a L\'{e}vy process (BDSDELs,
for short) initiated by Hu and Ren \cite{HY1}. Such a relation in a
sense could be viewed as an extension of the nonlinear Feynman-Kac
formula to stochastic PDIEs, which, to our best knowledge, is new.
Note also that this work could be considered as a generalization for
the updated result obtained by Ren and Otmani \cite{OR}, where the
authors treat deterministic PDIEs with nonlinear Neumann boundary
conditions.

The rest of this paper is organized as follows. In Section 2, we
introduced notion of
 stochastic viscosity solutions and all details associated. In Section 3, we review the
 generalized backward doubly stochastic differential equations driven by a L\'{e}vy process and
 its connection to stochastic PDIEs, from which the existence of the stochastic viscosity
  solution will follow.
\section{ Notion of viscosity solution for SPDIE}
\subsection{Notations,  assumptions and definitions}
\setcounter{theorem}{0} \setcounter{equation}{0} Let
$(\Omega,\mathcal{F}; \P)$ be a complete probability space on which
a $d$-dimensional Brownian motion $B=(B_t)_{t\geq 0}$ is defined .
Let ${\bf F}^{B}=\mathcal{F}_{t,T}^B$ denote the natural filtration
generated by $B$, augmented by the $\P$-null sets of $\mathcal{F}$.
Further, let ${\mathcal{M}}^{B}_{0,T}$ denote all the ${\bf
F}^{B}$-stopping times $\tau$ such $0\leq \tau\leq T$, a.s. and
${\mathcal{M}}^{B}_{\infty}$ be the set of all almost surely finite
${\bf F}^{B}$-stopping times. Let us introduce
$$\displaystyle{\ell^2=\Big\{x=(x^{(i)})_{i\geq 1};\;
\|x\|_{\ell^2}=(\sum_{i=1}^{\infty}|x^{(i)}|^2)^{1/2}<\infty\Big\}}.$$
For generic Euclidean spaces $E, E_{1}=\R^n$ or $\ell^2 $ and  we
introduce the following:
\begin{enumerate}
\item The symbol $\mathcal{C}^{k,n}([0,T]\times
E; E_{1})$ stands for the space of all $E_{1}$-valued functions
defined on $[0,T]\times E$ which are $k$-times continuously
differentiable in $t$ and $n$-times continuously differentiable in
$x$, and $\mathcal{C}^{k,n}_{b}([0,T]\times E; E_{1})$ denotes the
subspace of $\mathcal{C}^{k,n}([0,T]\times E; E_{1})$ in which all
functions have uniformly bounded partial derivatives.
\item For any sub-$\sigma$-field $\mathcal{G} \subseteq
\mathcal{F}_{T}^{B}$, $\mathcal{C}^{k,n}(\mathcal{G},[0,T]\times E;
E_{1})$ (resp.\, $\mathcal{C}^{k,n}_{b}(\mathcal{G},[0,T]\times E;
E_{1})$) denotes the space of all $\mathcal{C}^{k,n}([0,T]\times E;
E_{1})$  (resp.\, $\mathcal{C}^{k,n}_{b}([0,T]\times
E;E_{1}))$-valued random variable that are
$\mathcal{G}\otimes\mathcal{B}([0,T]\times E)$-measurable;
\item $\mathcal{C}^{k,n}({\bf F}^{B},[0,T]\times E; E_{1})$
(resp.$\mathcal{C}^{k,n}_{b}({\bf F}^{B},[0,T]\times E; E_{1})$) is
the space of all random fields $\phi\in
\mathcal{C}^{k,n}({\mathcal{F}}_{T},[0,T]\times E; E_{1}$ (resp.
$\mathcal{C}^{k,n}({\mathcal{F}}_{T},[0,T]\times E; E_{1})$, such
that for fixed $x\in E$ and $t\in [0,T]$, the mapping
$\displaystyle{\omega\mapsto \alpha(t,\omega,x)}$ is
${\bf F}^{B}$-progressively measurable.
\item For any sub-$\sigma$-field $\mathcal{G} \subseteq
\mathcal{F}^{B}$ and a real number $ p\geq 0$,
$L^{p}(\mathcal{G};E)$ to be all $E$-valued $\mathcal{G}$-measurable
random variable $\xi$ such that $ \E|\xi|^{p}<\infty$.
\end{enumerate}
Furthermore, regardless of their dimensions we denote by
$\left<\cdot,\cdot\right>$ and $|\cdot|$ the inner product and norm
in $E$ and $E_1$, respectively. For
$(t,x,y)\in[0,T]\times\R^{d}\times\R$, we denote
$D_{x}=(\frac{\partial}{\partial
x_{1}},....,\frac{\partial}{\partial x_{d}}),\,\\
D_{xx}=(\partial^{2}_{x_{i}x_{j}})_{i,j=1}^{d}$,
$D_{y}=\frac{\partial}{\partial y}, \,\
D_{t}=\frac{\partial}{\partial t}$. The meaning of $D_{xy}$ and
$D_{yy}$ is then self-explanatory.\newline Let $\Theta$ be an open
connected and smooth bounded domain of $\R^{n}\, (d\geq 1)$ such
that for a function $\psi\in\mathcal{C}^{2}_b(\R^{n}),\ \Theta$ and
its boundary  $\partial\Theta$ are characterized by
$\Theta=\{\psi>0\},\, \partial\Theta=\{\psi=0\}$ and, for any
$x\in\partial\Theta,\, \nabla\psi(x)$ is the unit normal vector
pointing towards the interior of $\Theta$.\newline Throughout this
paper, we shall make use of the following standing assumptions:
\begin{description}
\item $({\bf A1})$
\, The function $\sigma:\R^n\rightarrow\R^{n}$ is uniformly Lipschitz continuous,
with a Lipschitz constant $K>0$.
\item $({\bf A2})$ The function $f:\Omega\times [0,T]\times\overline{\Theta}\times\R\times\ell^2\rightarrow\R$
is a continuous random field such that for fixed $,(x,y,q),\,
f(\cdot,\cdot,x,y,\sigma^*q)$ is a
$\mathcal{F}^{B}_{t,T}$-measurable; and there exists a constant
$K>0$, for all $(t,x,y,z),\, (t',x',y',z')\in
[0,T]\times\R^n\times\R\times\ell^2,$ such that for $\P$-a.e.
$\omega$,\newline\newline $
\begin{array}{ll}
|f(\omega,0,0,0,0)|\leq K\\\\
|f(\omega,t,x,y,z)-f(\omega,t',x',y',z')|\leq
K(|t-t'|+|x-x'|+|y-y'|+|z-z'|).
\end{array}
$
\item $({\bf A3})$ The function $\phi:\Omega\times [0,T]\times\overline{\Theta}\times\R\rightarrow\R$
is a continuous random field such that, for fixed $(x,y),
\phi(\cdot,\cdot,x,y)$ is a $\mathcal{F}^{B}_{t,T}$-measurable; and
there exists a constant
 $K>0$,  for all $(t,x,y),\, (t',x',y')\in
[0,T]\times\R^n\times\R$, such that for $\P$-a.e.
$\omega$,\newline\newline $
\begin{array}{ll}
|\phi(\omega,0,0,0)|\leq K\\\\
|\phi(\omega,t,x,y)-\phi(\omega,t',x',y')|\leq
K(|t-t'|+|x-x'|+|y-y'|).
\end{array}
$
\item $({\bf A4})$ The function $u_0 :\R^n\rightarrow\R$ is continuous, for all  $x\in\R^n,$
 such that for some positive constants $K,\,p>0$,\\\\
$
\begin{array}{ll}
|u_0(x)|\leq K(1 + |x|^p).
\end{array}
$
\item $({\bf A5
})$ The function $g\in C^{0,2,3}_b ([0, T]\times\overline{\Theta}\times\R; \R^d)$.
\end{description}
As shown by the work of Buckdahn and Ma \cite{BM1,BM2}, our
definition of stochastic viscosity solution will depend heavily on
the following stochastic flow $\eta\in C({\bf F}^B, [0,
T]\times\R^n\times\R)$, defined as the unique solution of the
following stochastic differential equation in the Stratonovich
sense:
\begin{eqnarray}
\eta(t,x,y)&=&y+\int_t^T\langle g(s,x,\eta(s,x,y)),
\circ dB_s\rangle.\label{p1}
\end{eqnarray}
We refer the reader to \cite{BM1} for a lucid discussion on this
topic. Under the assumption $({\bf
A5})$, the mapping $y\mapsto \eta(t,x,y)$
 defines a diffeomorphism for all $(t,x),\; \P$-a.s. (see Protter \cite{Pr}). Let us denote its $y$-inverse
  by $\varepsilon(t,x,y)$. Then, one can show that $\varepsilon(t, x, y)$ is the solution to the following
  first-order SPDE:
\begin{eqnarray*}
\varepsilon(t,x,y)=y-\int_t^T\langle D_y\varepsilon(s,x,y),\, g(s,x,\eta(s,x,y))
\circ dB_s\rangle.
\end{eqnarray*}
We now define the notion of stochastic viscosity solution for SPDIEs
\eqref{i1}. In order to simply the notation, we denote:
\begin{eqnarray*}
A_{f,g}(\varphi(t,x))=L\varphi(t,x)+f(t,x,\varphi(t,x),(\varphi^1_k(t,x))_{k=1}^{\infty})
-\frac{1}{2}\langle g,D_yg\rangle(t,x,\varphi(t,x))
\end{eqnarray*}
and $\Psi(t,x)=\eta(t,x,\varphi(t,x))$
\begin{definition}
(1) A random field $u\in C({\bf F}^B, [0,
T]\times\overline{\Theta})$ is called a stochastic viscosity
subsolution of the SPDIEs \eqref{i1} if $u(T,x)\leq u_0(x)$, for all
$x\in \overline{\Theta}$
 and if for any stopping time $\tau\in{\mathcal{M}}^{B}_{0,T}$, any state variable $\xi\in L^{0}
 (\mathcal{F}^B_{\tau},[0,T]\times\Theta)$, and any random field $\varphi\in C^{1,2}(\mathcal{F}^B\tau,
 [0, T]\times\R^n)$ satisfying that
$$
\begin{array}{ll}
u\left(t,x\right)-\Psi\left(t,x\right)
 \leq 0=u\left(\tau(\omega),\xi(\omega)\right)-\Psi
\left(\tau(\omega),\xi(\omega)\right)
\end{array}$$
for all $(t, x)$ in a neighborhood of $(\xi,\tau)$, $\P$-a.e. on the set $\{0<\tau<T\}$, it holds that
\begin{description}
\item $(a)$ on the event $\{0<\tau<T\}$,
\begin{eqnarray*}
\mathrm{A}_{f,g}\left(\Psi\left(\tau,\xi\right)\right)- D_y\Psi
\left(\tau,\xi\right)D_t \varphi\left(\tau,\xi\right) \leq 0, \
\P\mbox{-a.e.};
\end{eqnarray*}

\item $(b)$ on the event $\{0<\tau<T\}\cap\{\xi\in\partial \Theta\}$,
\begin{align} \min\left\{\mathrm{A}_{f,
g}\left(\Psi\left(\tau,\xi\right)\right)- D_y \Psi
\left(\tau,\xi\right) D_t
\varphi\left(\tau,\xi\right),\,-\frac{\displaystyle{\partial
\Psi}}{\displaystyle{\partial
n}}\left(\tau,\xi\right)-\phi\left(\tau,\xi,\Psi\left(\tau,\xi\right)\right)
\right\} \leq 0, \ \P\mbox{-a.e.} \label{E:viscosity01}
\end{align}
\end{description}
(2) A random field $u\in C({\bf F}^B, [0,
T]\times\overline{\Theta})$ is called a stochastic viscosity
subsolution of the SPDIE $(f, g)$ \eqref{i1} if $u(T,x)\geq u_0(x)$,
for all $x\in \overline{\Theta}$ and if for any stopping time
$\tau\in{\mathcal{M}}^{B}_{0,T}$, any state variable $\xi\in L^{0}
(\mathcal{F}^B_{\tau},[0,T]\times\Theta)$, and any random field
$\varphi\in C^{1,2} (\mathcal{F}^B\tau, [0, T]\times\R^n)$
satisfying that$$
\begin{array}{ll}
u\left(t,x\right)-\Psi\left(t,x\right)
 \geq 0=u\left(\tau(\omega),\xi(\omega)\right)-\Psi
\left(\tau(\omega),\xi(\omega)\right)
\end{array}
$$
for all $(t, x)$ in a neighborhood of $(\xi,\tau)$, $\P$-a.e. on the set $\{0<\tau<T\}$, it holds that
\begin{description}
\item $(a)$ on the event $\{0<\tau<T\}$,
\begin{eqnarray*}
\mathrm{A}_{f,g}\left(\Psi\left(\tau,\xi\right)\right)- D_y\Psi
\left(\tau,\xi\right)D_t \varphi\left(\tau,\xi\right) \geq 0,\
\P\mbox{-a.e.};
\end{eqnarray*}
\item $(b)$ on the event $\{0<\tau<T\}\cap\{\xi\in\partial \Theta\}$,
\begin{align} \max\left\{\mathrm{A}_{f,
g}\left(\Psi\left(\tau,\xi\right)\right)- D_y \Psi
\left(\tau,\xi\right) D_t
\varphi\left(\tau,\xi\right),\,-\frac{\displaystyle{\partial
\Psi}}{\displaystyle{\partial
n}}\left(\tau,\xi\right)-\phi\left(\tau,\xi,\Psi\left(\tau,\xi\right)\right)
\right\} \geq 0, \P\mbox{-a.e.} \label{E:viscosity01}
\end{align}
\end{description}
(3) A random field $u \in \mathcal{C}\left(\mathbf{F}^B, [0,T]\times
\overline{\Theta}\right)$ is called a stochastic viscosity solution
of SPDIE $(f,g)$ \eqref{i1} if it is both a stochastic viscosity
subsolution and a a stochastic viscosity supersolution.
\end{definition}
\begin{remark}
We remark that if $f,\,\phi$ are deterministic and $g\equiv 0$, the flow $\eta$ becomes
$\eta(t,x,y)=y$ and $\Psi(t,x)=\varphi(t,x),\, \forall\ (t,x,y)\in[0,T]\times\R^n\times\R$. Thus,
definition $2.1$ coincides with the definition of
(deterministic) viscosity solution of PDIE $(f,0,\phi)$ given in \cite{OR}.
\end{remark}
Next, the following notion of a random viscosity solution will be a bridge linking the
stochastic viscosity solution and its
deterministic counterpart.
\begin{definition}
A random field $u\in C({\bf F}^B, [0, T]\times\overline{\Theta})$ is called an $\omega$-wise
 viscosity solution if for $\P$-almost all $\omega\in\Omega,\; u(\omega,\cdot,\cdot)$ is a
 deterministic viscosity solution of SPDIE $(f,0,\phi)$.
\end{definition}

\subsection{Doss-Sussmann transformation}
In this subsection, we study the Doss-Sussmann transformation. It
enables us to convert SPDIE $(f,g,\phi)$ to an SPDIE
$(\widetilde{f},0,\widetilde{\phi})$, where $\widetilde{f}$ and
$\widetilde{\phi}$ are well-defined random field depending on $f$,
$g$ and $\phi$ respectively. We get the following important result.
\begin{proposition}
Assume $({\bf A1})$--$({\bf A5})$ hold. A random field $u$ is a
stochastic viscosity sub- (resp. super)-solution to SPDIE $(f,
g,\phi)$ \eqref{i1} if and only if $v(\cdot,\cdot) =
\varepsilon(\cdot,\cdot, u(\cdot,\cdot))$ is a stochastic viscosity
sub-(resp. super)-solution to SPDIE $(\widetilde{f},
0,\widetilde{\phi})$, with
\begin{eqnarray}
&&\widetilde{f}(t,x,y,(z^{(k)})_{k=1}^{\infty})\nonumber\\
&=&\frac{1}{D_y\eta(t,x,y)}
\left[f\left(t,x,\eta(t,x,y),\left(D_y\eta(t,x,y)z^{(k)}+\sigma(x)D_x\eta(t,x,y){\bf 1}_{\{k=1\}}+\int_{\R}\theta^k(t,x,y,u)\nu(du)\right)^{\infty}_{k=1}\right)\right.\nonumber\\
&&\left.-\frac{1}{2}gD_yg(t,x,\eta(t,x,y))+L_x\eta(t,x,y)+\sigma(x)D_{xy}\eta(t,x,y)\left(z^{(1)}+\int_{\R}\theta^1(t,x,y,u)\nu(du)\right)\right.\nonumber\\
&&\left.+\frac{1}{2}D_{yy}\eta(t,x,y)\sum_{k=1}^{\infty}\left|z^{(k)}+\int_{\R}\theta^k(t,x,y,u)\nu(du)\right|^2\right]\label{Doss1}
\end{eqnarray}
and
\begin{eqnarray}
\widetilde{\phi}(t,x,y)=\frac{1}{D_y\eta(t,x,y)}\left[h(t,x,\eta(t,x,y))+D_x\eta(t,x,y)\nabla\psi(x)\right].\label{Doss2}
\end{eqnarray}
The process $\theta$ is defined by
\begin{eqnarray}
\theta^k(t,x,y,u)=[\eta(t,x+\sigma(x)u,y)-\eta(t,x,y)]p_k(u).\label{Def}
\end{eqnarray}
\end{proposition}
\begin{remark}
Let us recall that under the assumption $({\bf A5})$ the random
field $\eta$ belongs to $C^{0,2,2}(F^{B}, [0,T]\times\R^n\times\R)$,
and hence that the same is true for $\varepsilon$. Then, considering
the transformation $\Psi(t,x) = \eta(t,x,\varphi(t,x))$, we obtain
\begin{eqnarray*}
D_x\Psi&=& D_x\eta +D_y\eta D_x\varphi,\\
D_{xx}\Psi&=& D_{xx}\eta +2(D_{xy}\eta)(D_x\varphi)^{*}+(D_{yy}\eta)(D_x\varphi)(D_x\varphi)^* +(D_y\eta)(D_{xx}\varphi).
\end{eqnarray*}
Moreover, since for all $(t, x, y)\in[0,T]\times\R^n\times\R$ the equality $\varepsilon(t, x,\eta(t, x, y))= y$ holds $\P$-almost
surely, we also have
\begin{eqnarray*}
D_{x}\varepsilon+D_y\varepsilon D_x\eta&=&0,\\
D_y\varepsilon D_y\eta&=&1,\\
D_{xx}\varepsilon+2(D_{xy}\varepsilon(D_x\eta)^*+(D_{yy}\varepsilon)(D_{x}\eta)(D_x\eta)^* +(D_y\varepsilon)(D_{xx}\eta)&=&0,\\
(D_{xy}\varepsilon)(D_y\eta)+(D_{yy}\varepsilon)(D_x\eta)(D_y\eta)+(D_y\varepsilon)(D_{xy}\eta)&=&0,\\
(D_{yy}\varepsilon)(D_{y}\eta)^2+(D_{y}\varepsilon)(D_{yy}\eta)&=&0,
\end{eqnarray*}
where all the derivatives of the random field $\varepsilon(\cdot,\cdot,\cdot)$ are evaluated at $(t, x,\eta(t, x, y))$, and all
those of $\eta(\cdot,\cdot,\cdot)$ are evaluated at $(t, x, y)$.
\end{remark}
\begin{proof}[Proof of Proposition 2.4.]
We shall only prove that if $u\in C({\bf F}^B; [0,T]\times\R^n)$ is
a stochastic viscosity subsolution to SPDIEs $(f,g,\phi)$, then
$v(\cdot,\cdot) = \varepsilon(\cdot,\cdot, u(\cdot,\cdot))\in C({\bf
F}^B, [0, T]\times\R^n)$ is a stochastic viscosity subsolution to
SPDIE$(\widetilde{f},0,\widetilde{\phi})$. The remaining part can be
proved without enough difficulties in the similar way.

To this end, let $u\in C({\bf F}^B; [0,T]\times\R^n)$ be a
stochastic viscosity subsolution to SPDIEs $(f,g,\phi)$ and let
$v(t,x)=\varepsilon(t,x, u(t,x))$. Let us take
$\tau\in\mathcal{M}^B_{0,T}, \xi\in
L^2(\mathcal{F}^{B}_{\tau},\R^n)$ arbitrarily, and let $\varphi\in
C^{1,2}(\mathcal{F}^{B}_{\tau},\R^n)$ be such that
\begin{eqnarray*}
v(\omega,t,x)-\varphi(\omega,t,x)\leq 0=v(\omega,\tau(\omega),\xi(\omega))-\varphi(\omega,\tau(\omega),\xi(\omega))
\end{eqnarray*}
for all $(t, x)$ in a neighborhood of $(\xi,\tau)$, $\P$-a.e. on the set $\{0<\tau<T\}$.

Setting $\Psi(t,x)=\eta(t,x,\varphi(t,x))$ and since mapping $y \mapsto \eta(t,x,\varphi(t,x,y))$
is strictly increasing,
we have
\begin{eqnarray*}
u(t,x)-\Psi(t,x)&=&\eta(t,x,v(t,x))-\eta(t,x,\varphi(t,x))\\
&\leq& 0=\eta(\tau,\xi,v(\tau,\xi))-\eta(\tau,\xi,\varphi(\tau,\xi))=u(\tau,\xi)-\Psi(\tau,\xi),
\end{eqnarray*}
for all $(t, x)$ in a neighborhood of $(\xi,\tau)$, $\P$-a.e. on the
set $\{0<\tau<T\}$. Therefore, since $u$ is a stochastic viscosity
subsolution to SPDIE$(f,g,\phi)$, it follows that $\P$-a.e. on
$\{0<\tau<T\}$,
\begin{eqnarray}
\mathrm{A}_{f,g}\left(\Psi\left(\tau,\xi\right)\right)- D_y\Psi
\left(\tau,\xi\right)D_t \varphi\left(\tau,\xi\right)
\geq 0.\label{v1}
\end{eqnarray}
On the other hand, we have
\begin{eqnarray*}
L\Psi(t,x)&=&L_x\eta(t,x,\varphi(t,x))+D_y\eta(t,x,\varphi(t,x))
L\varphi(t,x)\\&&+\sigma(x)D_{xy}\eta(t,x,\varphi(t,x))(D_x\varphi(t,x))\\
&&+\frac{1}{2}D_{yy}\eta(t,x,\varphi(t,x))(D_x\varphi(t,x))^2,
\end{eqnarray*}
where $L_x$ is the same as the operator $L$, with all the derivatives taken with
respect to the second variable $x$ from which together with \eqref{Doss1}, we obtain
\begin{eqnarray*}
D_y\varepsilon(t,x,\Psi(t,x))A_{f,g}(\Psi(t,x))=A_{\widetilde{f},0}(\varphi(t,x)).
\end{eqnarray*}
Finally, in virtue of \eqref{v1}, we get
\begin{eqnarray*}
A_{\widetilde{f},0}(\varphi(\tau,\xi))\geq D_t\varepsilon(\tau,\xi).
\end{eqnarray*}
That is, part (a) of Definition 2.1. is established. To derive part
(b), noting that for all $(t, x)\in[0,T]\times\partial\Theta$, we
have
\begin{eqnarray*}
\frac{\partial\Psi}{\partial n}(t,x)&=& D_x\Psi(t,x)\cdot\nabla \psi(x)\\
&=& D_x \eta(t, x, \varphi(t, x))\cdot\nabla \psi(x) + D_y\eta(t, x,\varphi(t, x))D_x\varphi(t, x)\cdot\nabla\psi(x)\\
&=& D_x \eta(t, x, \varphi(t, x))\cdot\nabla \psi(x) + D_y\eta(t,
x,\varphi(t, x))\frac{\partial\varphi}{\partial n}(t,x).
\end{eqnarray*}
This shows that
\begin{eqnarray*}
\frac{\partial\Psi}{\partial n}(\tau,\xi)+\phi(\tau,\xi,\Psi(\tau,\xi))=D_x \eta(\tau, \xi, \varphi(t, x))
\left(\frac{\partial\varphi}{\partial n}(\tau,\xi)+\widetilde{\phi}
(\tau,\xi,\varphi(\tau,\xi))\right)
\end{eqnarray*}
where $\widetilde{\phi}$ is defined by \eqref{Doss2}. Because $D_y\eta(t,x,y)$ is strictly positive, we have $\P$-a.s.
on $\{0<\tau<T\}\cap \{\xi\in\partial\Theta\}$
\begin{eqnarray*}
\min\left\{A_{\widetilde{f},0}(\varphi(\tau,\xi))-D_t\varepsilon(\tau,\xi),-\frac{\partial\varphi}{\partial
n}(\tau,\xi)-\widetilde{\phi}
(\tau,\xi,\varphi(\tau,\xi))\right\}\geq 0.
\end{eqnarray*}
That is, $v$ is a stochastic viscosity subsolution of
SPDIE$(\widetilde{f},0,\widetilde{\phi})$.
\end{proof}
\section{Generalized BDSDELs and SPDIEs with Neumann
boundary condition}

The main object of this section is to show how a semi-linear SPDIE
$(f,g,\phi)$ \eqref{i1} is related to the so-called generalized
BDSDELs (GBDSDELs, for short) initiated by Hu and Ren\cite{HY1}, in
the Markovian case. To begin with, let us introduce another complete
probability space $(\Omega', \mathcal{F}',\P')$ on which we define a
L\'{e}vy process $L$ characterized by the following famous
L\'{e}vy-Khintchine formula
\begin{eqnarray*}
\E({\rm{e}}^{iuL_t})={\rm{e}}^{-t\Phi(u)}\;\; \mbox{with}\;\;
\Phi(u)=-ibu+\frac{\sigma^2}{2}u^2-\int_{\R}\left({\rm{e}}^{iuy}-1-iuy
{\bf 1}_{\{|y|\leq 1\}}\right)\nu(dy).
\end{eqnarray*}
Thus $L$ is characterized by its L\'{e}vy triplet $(b,\sigma,\nu)$
where $b\in\R, \sigma^2\geq 0$ and  $\nu$ is a measure defined in
$\R\backslash \{0\}$ which satisfies that\newline $
\begin{array}{l}
{\rm(i)} \ \int_{\R}(1 \wedge y^2)\nu(dy) < +\infty,\\
{\rm(ii)} \ \exists\ \varepsilon>0\, \mbox{and}\, \lambda>0
\,\mbox{such that}\,
\int_{(-\varepsilon,\varepsilon)^{c}}{\rm{e}}^{\lambda|y|}\nu(dy)<+\infty.
\end{array}
$
\newline
This implies that the random variable $L_t$ have moment of all orders, i.e\, $m_1=\E(L_1)=b+\int_{|y|\geq 1}y\nu(dy)$
and
$
\begin{array}{cc}
m_i=\int_{-\infty}^{+\infty}y^i\nu(dy)<\infty,\;\; \forall\, i\geq 2.
\end{array}
$ For the background on L\'{e}vy processes, we refer the reader to
\cite{JB,KS}.

We define the following family of $\sigma$-fields:
\begin{eqnarray*}
\mathcal{F}_{t}^{L}=\sigma(L_{r}-L_s, s\leq r \leq t)\vee
\mathcal{N}',
\end{eqnarray*}
where $\mathcal{N}'$ denotes all the $\P'$-null sets in $\mathcal{F}'$. Denote ${\bf F}^{L}=(\mathcal{F}_{t}^{L})
_{0\leq t\leq T}$.

Next, we
consider the product space $(\bar{\Omega}, \bar{\mathcal{F}}, \bar{\P})$ where
\begin{eqnarray*}
\bar{\Omega}=\Omega\otimes\Omega';\;\;\;\;\; \bar{\mathcal{F}}=\mathcal{F}\otimes\mathcal{F}',\;\;;\;\;\;\;\;
\bar{\P}=\P\otimes\P',
\end{eqnarray*}
and define $\mathcal{F}_{t}=\mathcal{F}_{t,T}^{B}\otimes\mathcal{F}_{t}^{L}
$ for all $t\in [0,T]$. We remark that ${\bf
F}= \{\mathcal{F}_{t},\ t\in [0,T]\}$ is neither increasing nor
decreasing so that it does not a filtration. Further, we assume that
random variables $\xi(\omega),\;\omega\in \Omega$ and
$\zeta(\omega'),\; \omega'\in \Omega'$ are considered as random
variables on $\bar{\Omega} $ via the following identifications:
\begin{eqnarray*}
\xi(\omega,\omega')=\xi(\omega);\,\,\,\,\,\zeta(\omega,\omega')=\zeta(\omega').
\end{eqnarray*}

We denote by $(H^{(i)})_{i\geq 1}$ the Teugels Martingale associated with the L\'{e}vy
process $\{L_t : t\in[0, T]\}$. More precisely
\begin{eqnarray*}
H^{(i)}=c_{i,i}Y^{(i)}+c_{i,i-1}Y^{(i-1)}+\cdot\cdot\cdot+c_{i,1}Y^{(1)}
\end{eqnarray*}
where $Y^{(i)}_{t}=L_{t}^{i}-m_i$ for all $i\geq 1$ with $L^{i}_t$ a power-jump process. That is $L^{1}_t=L_t$
and $L^{i}_t=\sum_{0<s<t}(\Delta L_s)^i$ for all $i\geq 2$, where $X_{t^-} = \lim_{s\nearrow t} X_s$ and
$\Delta X_t = X_t - X_{t^-}$. It was shown in Nualart and Schoutens \cite{NS1} that the coefficients $c_{i,k}$
correspond to the orthonormalization of the polynomials $1, x, x^2,...$ with respect to the measure
$\mu(dx)=x^{2}d\nu(x)+\sigma^{2}\delta_{0}(dx)$:
\begin{eqnarray*}
q_{i-1}(x)=c_{i,i}x^{i-1}+c_{i,i-1}x^{i-2}+\cdot\cdot\cdot+c_{i,1}.
\end{eqnarray*}
We set
\begin{eqnarray*}
p_i(x)=xq_{i-1}(x)=c_{i,i}x^{i}+c_{i,i-1}x^{i-1}+\cdot\cdot\cdot+c_{i,1}x^{1}.
\end{eqnarray*}
The martingale $(H^{(i)})_{i=1}^{\infty}$ can be chosen to be pairwise strongly orthonormal
martingale.

We consider the following spaces of processes:
\begin{enumerate}
\item $\mathcal{M}^{2}(\ell^{2})$ denotes the space of $\ell^2$-valued, square integrable and $\mathcal{F}_{t}$
-measurable processes $\varphi=\{\varphi_{t}:\;
t\in[0,T]\}$ such that
\begin{description}
\item $\displaystyle \|\varphi\|^{2}_{{\mathcal{M}}^{2}}=\E\int ^{T}_{0}\|\varphi_{t}\|^{2}dt<\infty$.
\end{description}
\item  $\mathcal{S}^{2}(\R)$ is the subspace of  $\mathcal{M}^{2}(\R)$
formed by the $\mathcal{F}_{t}$-measurable, right continuous with left
limit (rcll) processes $\varphi=\{\varphi_{t}:\; t\in[0,T]\}$ such
that
\begin{description}
\item $\displaystyle{\|\varphi\|^{2}_{\mathcal{S}^{2}}=\E\left(\sup_{0\leq t\leq T}|\varphi_{t}|
^{2}\right)<\infty}$.
\end{description}
\end{enumerate}
Finally, let
$\mathcal{E}^{2}=\mathcal{S}^{2}(\R)\times{\mathcal{M}}^{2}(\ell^{2})$
be endowed with the norm
\begin{eqnarray*}
\|(Y,Z)\|^{2}_{\mathcal{E}^2}=\E\left(\sup_{0\leq t\leq T}|Y_{t}|
^{2}+\int ^{T}_{0}\|Z_{t}\|^{2}dt\right).
\end{eqnarray*}
\subsection{A class of reflected diffusion process and GBDSDELs}
We now introduce a class of reflected diffusion process. Let
$\Theta$ be a regular convex and bounded
 subsect of $\R^n$ , which is such that for a function $\psi\in C^{2}_b(\R^n),\; \Theta=\{x\in\R^n:\,
 \psi(x)>0\},\; \partial\Theta=\{x\in\R^n:\, \psi(x)=0\}$ and for all $x\in\partial\Theta,\; \nabla\psi(x)$
 coincides with the unit normal pointing towards the interior of $\Theta$ (see \cite{LS}). Under assumption
  $({\bf A1})$, we know from \cite{LS} that for every $(t,x)\in [0,T]\times\overline{\Theta}$ there exists a
   unique pair of progressively measurable process $(X^{t,x}_s, A^{t,x}_s)_{t\leq s\leq T}$, which is a
   solution to the following reflected SDE:
\begin{eqnarray}
\left\{
\begin{array}{l}
\displaystyle \P(X_s^{t,x}\in \overline{\Theta},\, s\geq t)=1\\\\
\displaystyle
X_s^{t,x}=x+\int_{t}^s\sigma(X_{r^-}^{t,x})dL_r+\int^s_t
\nabla\psi(X^{t,x}_s)dA^{t,x}_s,\; s\geq t,
\end{array}\right.
\label{RSDEJ1}
\end{eqnarray}
 where $A^{t,x}_s=\int^{s}_t {\bf 1}_{\{X_{r}^{t,x}\in\partial\Theta\}}dA^{t,x}_r,\;\; A^{t,x}$ is
 an increasing process with bounded variation on $[0,T],\, 0<T<\infty,\; A_0=0$. Furthermore,
 we have the following proposition.
\begin{proposition}\label{P:continuity00}
There exists a constant $ C>0 $ such that for all $x,\,x'\in
\overline{\Theta}$,
\begin{eqnarray*}
\mathbb{E}\left[\sup_{0\leq s\leq
T}|X^{x}_{s}-X^{x'}_{s}|^4\right] \leq
C|x-x'|^4
\end{eqnarray*}
and
\begin{eqnarray*}
\mathbb{E}\left[\sup_{0\leq s\leq
T}|A_s^{x}-A_{s}^{x'}|^4\right]\leq C|x-x'|^4.
\end{eqnarray*}
\end{proposition}
The main subject in this section is the following GBDSDELs, for $(t,
x)\in[0; T]\times\R^n$,
\begin{eqnarray}
\left\{
\begin{array}{l}
{\rm(i)}\; \displaystyle \E\left[\sup_{t\leq s\leq T}|Y^{t,x}_s|^2+\int_{t}^{T}\|Z^{t,x}_s\|^2ds\right]<\infty ;\\
{\rm(ii)}\;\displaystyle
Y_{s}^{t,x}=u_0(X_T^{t,x})+\int_{s}^{T}f(r,X_{r}^{t,x},Y_{r}^{t,x},Z_{r}^{t,x})dr+\int_{s}^{T}
\phi(r,X_{r}^{t,x},Y_{r}^{t,x})dA^{t,x}_r\\ \displaystyle
\;\;\;\;\;\;\;\;\;\;\;\;\;\;\;\;+\int_{s}^{T}g(r,X_{r}^{t,x},Y_{r}^{t,x})
\,dB_{r}-\sum_{i=1}^{\infty}\int_{s}^{T}(Z^{t,x}_{r})^{(i)}dH^{(i)}_{r},\,\
t\leq s\leq T.
\end{array}\right.
\label{GBDSDEmarkovian}
\end{eqnarray}
\begin{remark}
In what follows, we will assume $n=1$. The multidimensional case can
be completed without major difficulties.
\end{remark}
Let us recall an existence and uniqueness result appear in
\cite{HY1} and a generalized version of the Itô-Ventzell formula
whose proof is analogous to the corresponding one in Buckdahn-Ma
\cite{BM1} replacing the Brownian motion $W$ by the Teugels
martingale $(H^{(i)})_{i\geq 1}$.

\begin{theorem}
Assume that $({\bf A1})$--$({\bf A5})$ hold. For each $(t, x)\in[0,
T]\times\R$, GBDSDEL \eqref{GBDSDEmarkovian} has a unique solution
$(Y^{t,x}, Z^{t,x})\in\mathcal{E}^2$.
\end{theorem}
\begin{theorem}\label{theorem3.4}
Suppose that $M\in C^{0,2}({\bf F},[0, T]\times\R)$ is a semimartingale
in the sense that for every spatial parameter $x\in\R$ the process $t\mapsto M (t, x),\, t\in[0, T]$, is of the form:
\begin{eqnarray*}
M (t, x) = M (0, x)+\int^t_0
G(s, x)ds+\int^t_0\langle N (s, x),{dB}_s\rangle+\sum_{i=1}^{\infty}\int^t_0 K^{(i)} (s, x)dH^{(i)}_s,
\end{eqnarray*}
where $G\in C^{0,2}({\bf F}^B, [0, T]\times\R),\; N\in C^{0,2}({\bf F}^B, [0, T]\times\R;\R^d)$, and
the process $K$ belongs to $C^{0,2}({\bf F}^L, [0, T]\times\R;\ell^2)$. We also consider
the process $\alpha\in C ({\bf F}, [0, T])$
of the form
\begin{eqnarray*}
\alpha_t = \alpha_0+\int^t_0 \beta_sds+\int^t_0 \theta_sdA_s+\int^t_0\gamma_s{dB}_s
+\sum_{i=1}^{\infty}\int^t_0 \delta_s^{(i)}dH^{(i)}_s
\end{eqnarray*}
where $\beta, \theta\in \mathcal{S}^2(\R),\; \gamma\in \mathcal{M}^2(\R^{d})$, and
$\delta\in \mathcal{M}^2(\ell^2)$. Then the following equality holds $\P$-almost surely
for all $0\leq t\leq T$:
\begin{eqnarray*}
M (t, \alpha_t)& =& M (0, \alpha_0) + \int^t_0 G(s, \alpha_s)ds +
\int^t_0 \langle N(s, \alpha_s), dB_s\rangle
+ \sum_{i=1}^{\infty}\int^t_0 K^{(i)} (s, \alpha_s)dH^{(i)}_s\\
&&+ \int^t_0D_x M(s,\alpha_s)\beta_sds+\int^t_0D_x
M(s,\alpha_s)\theta_s dA_s
+ \int^t_0 \langle D_x M(s,\alpha_s),\gamma_s dB_s\rangle \\
&&+\sum_{i=1}^{\infty}\int^t_0 D_x
M(s,\alpha_s)\delta^{(i)}_sdH^{(i)}_s
-\frac{1}{2}\sum_{i=1}^{d}\int^t_0 D_{xx}
M(s,\alpha_s)|\gamma^{i}_s|^2ds
\\&&+\frac{1}{2}\sum_{i=1}^{\infty}\int^t_0 D_{xx} M(s,\alpha_s)|\delta_s^{(i)}|^2ds
+ \sum_{i=1}^{\infty}\int^t_0 D_x K^{(i)}(s,\alpha_s)\delta^{(i)}_sds
\\
&&-\sum_{i=1}^{d}\int^t_0 D_x N^{i}(s,\alpha_s)\gamma^{i}_sds.
\end{eqnarray*}
\end{theorem}
\subsection{Existence of stochastic viscosity solution}

In this section we prove the existence of the stochastic viscosity
solution to the SPDIEs $(f, g, \phi)$. Our main idea is to apply the
Doss transformation to the GBDSDEL \eqref{GBDSDEmarkovian} to obtain
resulting GBDSDEL without the stochastic integral against $dB$,
which naturally become a GBSDEL with new generators being exactly
$\widetilde{f}$ and $\widetilde{\phi}$. For this, for each
$(t,x)\in[0,T]\times\R, t\leq s\leq T$, let us define the following processes,
\begin{eqnarray}
U^{t,x}_s&=&\varepsilon(s,X_{s}^{t,x},Y_{s}^{t,x}),\nonumber\\
(V^{(1)})^{t,x}_s&=& D_y\varepsilon(s,X_{s}^{t,x},Y_{s}^{t,x})(Z^{(1)})^{t,x}_s+\sigma(X_{s}^{t,x})D_x
\varepsilon(s,X_{s}^{t,x},Y_{s}^{t,x})\nonumber\\
&&+\int_{\R}[\varepsilon(s,X_s^{t,x}+\sigma(X_s^{t,x})u,Y_s^{t,x})
-\varepsilon(s,X_{s}^{t,x},Y_s^{t,x})-D_x\varepsilon(s,X_{s}^{t,x},Y_s^{t,x})\sigma(X_{s}^{t,x})u]p_1(u)\nu(du),\nonumber\\
(V^{(k)})^{t,x}_s&=&
D_y\varepsilon(s,X_{s}^{t,x},Y_{s}^{t,x})(Z^{(k)})^{t,x}_s\label{D1}\\
&&+\int_{\R}[\varepsilon(s,X_s^{t,x}+\sigma(X_s^{t,x})u,Y_s^{t,x})
-\varepsilon(s,X_{s}^{t,x},Y_s^{t,x})-D_x\varepsilon(s,X_{s}^{t,x},Y_s^{t,x})\sigma(X_{s}^{t,x})u]p_k(u)\nu(du),\nonumber\\
&&k\in\{2,\cdot\cdot\cdot\}.\nonumber
\end{eqnarray}
From Proposition 3.4 appeared in \cite{BM1}, the process
$\{(U^{t,x}_s,V^{t,x}_s),\;\; s\in[t,T]\}$ belongs to $\mathcal{E}$
for each $(t,x)\in[0,T]\times\overline{\Theta}$.

Now we are ready to give the following result.
\begin{theorem}
For each $(t,x)\in[0,T]\times\overline{\Theta}$, the pair
$(U^{t,x},V^{t,x})$ is the unique solution of the following GBSDEL:
\begin{eqnarray}
U_{s}^{t,x}&=&u_0(X^{t,x}_{T})+\int_{t}^{T}\widetilde{f}(r,X^{t,x}_{r},U^{t,x}_r,V^{t,x}_r)dr
+\int_{s}^{T}\widetilde{\phi}(r,X^{t,x}_r,U^{t,x}_r)dA_r^{t,x}
\nonumber\\&&-\sum_{k=1}^{\infty}\int_{s}^{T}(V^{t,x}_{r})^{(k)}dH^{(k)}_{r},\;\; t\leq s\leq T,\nonumber\\
\label{a11}
\end{eqnarray}
where $\widetilde{f}$ and $\widetilde{\phi}$ are given by \eqref{Doss1} and \eqref{Doss2} respectively.
\end{theorem}

\begin{proof}
For clarity, $(X^{t,x}, Y^{t,x}, Z^{t,x}, U^{t,x}, V^{t,x})$ will be
replaced by $(X, Y, Z, U, V)$ throughout this proof. As it is shown
in \cite{BM1}, the mapping $(X, Y, Z)\mapsto (X, U, V)$ is
one-to-one, with the inverse transformation:
\begin{eqnarray}
Y_s&=&\eta(s,X_s,U_s),\nonumber\\
Z_s^{(1)}&=&D_y\eta(s,X_s,U_s)V^{(1)}_s+\sigma(X_s)D_x\eta(s,X_s,U_s)\nonumber\\
&&+\int_{\R}[\eta(s,X_s+\sigma(X_s)u,U_s)
-\eta(s,X_{s},U_s)-D_x\eta(s,X_{s},U_s)\sigma(X_{s})u]p_1(u)\nu(du),\nonumber\\
Z_s^{(k)}&=&D_y\eta(s,X_s,U_s)V^{(k)}_s\label{D2}\\
&&+\int_{\R}[\eta(s,X_s+\sigma(X_s)u,U_s)
-\eta(s,X_{s},U_s)-D_x\eta(s,X_{s},U_s)\sigma(X_{s})u]p_k(u)\nu(du),  k\in\{2,\cdot\cdot\cdot\}.\nonumber
\end{eqnarray}
Thanks to \eqref{D1} and \eqref{D2}, the uniqueness of GBSDEL
\eqref{a11} follows from GBDSDEL \eqref{GBDSDEmarkovian}. Thus, the
proof reduces to show that $(U, V)$ is a solution of the GBSDEL
\eqref{a11}. To this end, let us remark that $U_T = Y_T = u_0(X_T)$.
Moreover, applying the generalized Itô-Ventzell formula (see Theorem
4.2) to $\varepsilon(s, X_s,Y_s)$, and after a little calculation we
obtain
\begin{eqnarray}
U_t&=&u_0(X_T)+\int_t^TD_y\varepsilon(s,X_s,Y_s)f(s,X_s,Y_s,Z_s)ds
+\int_t^TD_y\varepsilon(s,X_s,Y_s)\phi(s,X_s,Y_s)dA_s\nonumber\\
&&-\sum_{k=1}^{\infty}\int_t^TD_y\varepsilon(s,X_s,Y_s)Z^{(k)}_sdH^{(k)}_s
-m_1\int_t^TD_x\varepsilon(s,X_s,Y_s)\sigma(X_s)ds
\nonumber\\
&&-\int_t^TD_x\varepsilon(s,X_s,Y_s)\sigma(X_s)dH^{(1)}_s\nonumber\\
&&-\int_t^TD_x\varepsilon(s,X_s,Y_s)\nabla\psi(X_s)dA_s
-\frac{1}{2}\int_t^T\sigma(X_s)^{*}D_{xx}\varepsilon(s,X_s,Y_s)\sigma(X_s)ds
\nonumber\\
&&-\sum_{k=1}^{\infty}\int_t^T\int_{\R}[\varepsilon(s,X_s+\sigma(X_s)u,Y_s)
-\varepsilon(s,X_{s},Y_s)-D_x\varepsilon(s,X_s,Y_s)\sigma(X_{s})u]p_k(u)\nu(du)dH^{(k)}_s\nonumber\\
&&+\int_t^T\int_{\R}[\varepsilon(s,X_s+\sigma(X_s)u,Y_s)-\varepsilon(s,X_{s},Y_s)
-D_x\varepsilon(s,X_s,Y_s)\sigma(X_{s})u]\nu(du)ds\nonumber\\
&&+\frac{1}{2}\sum_{k=1}^{\infty}\int_t^TD_{yy}\varepsilon(s,X_s,Y_s)|Z^{(k)}_s|^2ds
-\int_t^T\sigma^*(X_s)D_{xy}\varepsilon(s,X_s,Y_s)Z^{(1)}_sds\nonumber\\
&&-\frac{1}{2}\int_t^TD_y\varepsilon(s,X_s,Y_s)\langle g,D_yg\rangle(s,X_s,Y_s)ds.\label{ItV1}\\
U_t&=&u_0(X_T)+\int_t^TF(s,X_s,Y_s,Z_s)ds+
\int_t^T\Phi(s,X_s,Y_s)dA_s-\sum_{k=1}^{\infty}\int_t^TV^{(k)}dH_s^{(k)},\nonumber
\end{eqnarray}
where
\begin{eqnarray}
F(s,x,y,z)&=&D_y\varepsilon f(s,x,y,z)-m_1D_x\varepsilon\sigma(x)
+\frac{1}{2}D_{yy}\varepsilon\|z\|^2-\sigma^*(x)D_{xy}\varepsilon
z^{(1)}
\nonumber\\
&&-\frac{1}{2}\sigma^*(x)D_{xx}\varepsilon\sigma(x)-\frac{1}{2}D_y\varepsilon\langle
g,D_yg\rangle(s,x,y)\nonumber\\
&&
+\int_{\R}[\varepsilon(s,x+\sigma(x)u,y)-\varepsilon(s,x,y)-D_x\varepsilon\sigma(x)u]\nu(du)\label{F}
\end{eqnarray}
and
\begin{eqnarray}
\Phi(s,x,y)=D_y\varepsilon\phi(s,x_s,y,z)-D_{x}\varepsilon\nabla\psi(x)\label{Phi},
\end{eqnarray}
replaced $\varepsilon(s,x,y)$ by $\varepsilon$. Comparing
\eqref{ItV1} with \eqref{a11}, it suffices to show that
\begin{eqnarray}
F(s,X_s,Y_s,Z_s)=\widetilde{f}(s,X_s,U_s,V_s),\;\;\; \forall\,
s\in[0,T],\; \P\mbox{-a.s.}\label{F0}
\end{eqnarray}
and
\begin{eqnarray}
\Phi(s,X_s,Y_s)=\widetilde{\phi}(s,X_s,U_s),\;\; \forall\, s\in[0,T],\; \P\mbox{-a.s}.\label{Phi0}
\end{eqnarray}
To this end, if we write $\sigma(X_s) =\sigma_s$ and recall \eqref{Def} together with Remark 2.5 we obtain the following equalities:
\begin{eqnarray}
D_x\varepsilon(s,X_s,Y_s)\sigma(X_s)&=&-D_y\varepsilon(s,X_s,Y_s)\sigma(X_s)D_x\eta(s,X_s,U_s)\label{CF1}
\end{eqnarray}
\begin{eqnarray}
(D_y\varepsilon) f\left(s,X_s,Y_s,(Z^{(k)}_s)^{\infty}_{k=0}\right)&=& (D_y\varepsilon)
 f\Big(s,X_s,\eta(s,X_s,U_s),\big(D_y\eta V^{(k)}_s+\sigma^*_s(D_x\eta){\bf 1}_{\{k=1\}}\nonumber\\
&&+\int_{\R}\theta^{k}(s,X_s,U_s,u)\nu(du)\Big)^{\infty}_{k=1}\Big)\nonumber\\
\sigma^*_s(D_{xy}\varepsilon) Z_s^{(1)}&=&\sigma^*_s(D_{xy}\varepsilon) D_y\eta(s,X_s,U_s)V^{(1)}+\sigma^*_s(D_{xy}\varepsilon)\sigma^*_sD_x\eta\nonumber\\
&&+\sigma^*_s(D_{xy}\varepsilon)\int_{\R}\theta^{1}(s,X_s,U_s,u)\nu(du)
\label{CF2bis}
\end{eqnarray}
\begin{eqnarray}
-\frac{1}{2}(D_{yy}\varepsilon)\sum^{\infty}_{k=1}|Z^{(k)}_s|^2&=&\frac{1}{2}(D_y\varepsilon)(D_{yy}\eta)\sum^{\infty}_{k=1}|V^{(k)}_s|^2
+(D_{y}\varepsilon)^2(D_{yy}\eta)V^{(1)}_s\sigma_s(D_x\eta)\nonumber\\
&&+\frac{1}{2}(D_y\varepsilon)(D_{yy}\eta)|\sigma_s(D_x\eta)(D_y\varepsilon)|^2+(D_{y}\varepsilon)^2D_{yy}\eta \sum_{k=1}^{\infty}V^{(k)}_s\int_{\R}\theta^{k}(s,X_s,U_s,u)\nu(du)\nonumber\\
&&+\frac{1}{2}(D_y\varepsilon)(D_{yy}\eta)\sum_{k=1}^{\infty}\left|D_y\varepsilon\int_{\R}\theta^{k}(s,X_s,U_s,u)\nu(du)\right|^2\nonumber\\
&&+(D_y\varepsilon)(D_{yy}\eta)(D_y\varepsilon)^2\sigma_sD_x\eta\int_{\R}\theta^{1}(s,X_s,U_s,u)\nu(du).\nonumber\\
\label{CF3bis}
\end{eqnarray}
Hence plugging \eqref{CF1}-\eqref{CF3bis} in \eqref{F}, we get
\begin{eqnarray}
F(s,X_s,Y_s,Z_s)&=&D_y\varepsilon\Bigg[f\left(s,X_s,\eta,\left(D_y\eta
V^{(k)}_s+\sigma^*_s(D_x\eta){\bf 1}_{\{k=1\}}+\int_{\R}\theta^k(s,X_s,U_s,u)\nu(du)\right)^{\infty}_{k=0}\right)\nonumber\\
&&+m_1\sigma_sD_x\eta+\frac{1}{2}(D_{yy}\eta)\sum^{\infty}_{k=1}\left|V^{(k)}_s+D_y\varepsilon\int_{\R}\theta^{k}(s,X_s,U_s,u)\nu(du)\right|^2
-\frac{1}{2}\langle g,D_yg\rangle(s,X_s,\eta)\Bigg]\nonumber\\
&&+V^{(1)}\sigma^*_s\Big[(D_x\eta)(D_{y}\varepsilon)^2(D_{yy}\eta)-D_y\eta(D_{xy}\varepsilon)\Big]\nonumber\\
&&+\left(\int_{\R}\theta^{1}(s,X_s,U_s,u)\nu(du)\right)\sigma^*_s\Big[(D_x\eta)(D_{y}\varepsilon)^2(D_{yy}\eta)-D_y\eta(D_{xy}\varepsilon)\Big]\nonumber\\
&&\Bigg[\frac{1}{2}(D_{yy}\eta)(D_y\varepsilon|\sigma_s(D_x\eta)(D_y\varepsilon)|^2
-\sigma^*_s(D_{xy}\varepsilon)\sigma^*_sD_x\eta\label{CF2}\\
&&-\frac{1}{2}\sigma^*(x)D_{xx}\varepsilon\sigma_s
+\int_{\R}[\varepsilon(s,X_s+\sigma_su,Y_s)-\varepsilon(s,X_{s},Y_s)-(D_x\varepsilon)\sigma_{s}u]\nu(du)\Bigg],
\nonumber
\end{eqnarray}
where all the derivatives of the random field
$\varepsilon(\cdot,\cdot , \cdot)$ are to be evaluated at the point
$(s, x,\eta(s, x, y))$, and all those of $\eta(\cdot,\cdot , \cdot)$
at $(s, x, y)$. On other hand, using again Remark 2.5, we have
\begin{eqnarray}
-\frac{1}{2}\sigma^*(x)(D_{xx}\varepsilon)\sigma_s&=&(\sigma_s)^2D_{xy}\varepsilon D_x\eta-\frac{1}{2}(D_y\varepsilon)D_{yy}\eta|\sigma_sD_x\eta D_y\varepsilon|^2\nonumber\\
&&+\frac{1}{2}(D_y\varepsilon)(\sigma_s)^2(D_{xx}\eta)\label{sigma1}
\end{eqnarray}
and
\begin{eqnarray}
D_x\eta(D_y\varepsilon)^2(D_{yy}\eta)-D_{xy}\varepsilon
D_y\eta=D_y\varepsilon
D_{xy}\eta.\label{sigma2}
\end{eqnarray}
The equalities in \eqref{sigma1} and \eqref{sigma2}), together with
$D_y\varepsilon(s,X_s,Y_s )=(D_y\eta)^{-1}(s,X_s,U_s )$, imply that
\begin{eqnarray*}
F(s,X_s,Y_s,Z_s)&=&D_y\varepsilon\Big[f\left(s,X_s,\eta,\left(D_y\eta
V^{(k)}_s+\sigma^*_s(D_x\eta){\bf 1}_{\{k=1\}}+\int_{\R}\theta^k(s,X_s,U_s,u)\nu(du)\right)^{\infty}_{k=0}\right)\nonumber\\
&&+m_1\sigma_sD_x\eta+\frac{1}{2}(D_{yy}\eta)\sum^{\infty}_{k=1}\left|V^{(k)}_s+D_y\varepsilon\int_{\R}\theta^{k}(s,X_s,U_s,u)\nu(du)\right|^2
-\frac{1}{2}\langle g,D_yg\rangle(s,X_s,\eta)\Big]\nonumber\\
&&+\frac{1}{2}(D_y\varepsilon)\sigma_s^2(D_{xx}\eta)
+(D_y\varepsilon)\sigma_sD_{xy}\eta\left(V_s^{(1)}+\int_{\R}\theta^{1}(s,X_s,U_s,u)\nu(du)\right)\\
&&+\int_{\R}[\varepsilon(s,X_s+\sigma_su,Y_s)
-\varepsilon(s,X_{s},Y_s)-(D_x\varepsilon)\sigma_{s}u]\nu(du).
\nonumber
\end{eqnarray*}
Next, using again Remark 2.5 together with changing variable
($t=-u$), we have
\begin{eqnarray*}
&&\int_{\R}[\varepsilon(s,X_s+\sigma_su,Y_s)-\varepsilon(s,X_{s},Y_s)-(D_x\varepsilon)\sigma_{s}u]\nu(du)\\
&=&-D_y\varepsilon\int_{\R}[\eta(s,X_s+\sigma_su,U_s)-\eta(s,X_{s},U_s)-(D_x\eta)\sigma_{s}u]\nu(du)\\
&=&D_y\varepsilon\int_{\R}[\eta(s,X_s+\sigma_su,U_s)-\eta(s,X_{s},U_s)-(D_x\eta)\sigma_{s}u]\nu(du).
\end{eqnarray*}
Finally, we obtain
\begin{eqnarray}
F(s,X_s,Y_s,Z_s)&=&D_y\varepsilon\Big[f\left(s,X_s,\eta,\left(D_y\eta
V^{(k)}_s+\sigma^*_s(D_x\eta){\bf 1}_{\{k=1\}}+\int_{\R}\theta^k(s,X_s,U_s,u)\nu(du)\right)^{\infty}_{k=0}\right)\nonumber\\
&&+m_1\sigma_sD_x\eta+\frac{1}{2}(D_{yy}\eta)\sum^{\infty}_{k=1}\left|V^{(k)}_s+\int_{\R}\theta^{k}(s,X_s,U_s,u)\nu(du)\right|^2
-\frac{1}{2}\langle g,D_yg\rangle(s,X_s,\eta)\Big]\nonumber\\
&&+\frac{1}{2}(D_y\varepsilon)\sigma_s^2(D_{xx}\eta)
+(D_y\varepsilon)\sigma_sD_{xy}\eta\left(V_s^{(1)}+\int_{\R}\theta^{1}(s,X_s,U_s,u)\nu(du)\right)\label{CF3}\\
&&+D_y\varepsilon\int_{\R}[\eta(s,X_s+\sigma_su,U_s)-\eta(s,X_{s},U_s)-(D_x\eta)\sigma_{s}u]\nu(du).\nonumber
\end{eqnarray}
Since the expressions in \eqref{CF2} and \eqref{CF3} are equal, this
shows the equality in \eqref{F0}.
\\
Next, we show the
equality in \eqref{Phi0}. In fact,
\begin{eqnarray}
\Phi(s,X_s,Y_s)&=&D_y\varepsilon\phi(s,X_s,Y_s)-D_{x}\varepsilon\nabla\psi(X_s)\nonumber\\
&=&D_y\varepsilon(s,X_s,Y_s)[D_x\eta(s,X_s,U_s)\nabla\psi(X_s)+\phi(s,X_s,\eta(s,X_s,U_s)]\nonumber\\
&=&\frac{1}{D_y\eta(s,X_s,U_s)}[D_x\eta(s,X_s,U_s)\nabla\psi(X_s)+\phi(s,X_s,\eta(s,X_s,U_s)]\nonumber\\
&=&\widetilde{\phi}(s,X_s,U_s). \label{Phi}
\end{eqnarray}
This ends the proof of theorem.
\end{proof}
We are now ready to prove the existence of the stochastic viscosity
solutions of SPDIE $(f,g,\phi)$. Let us define for each $(t,
x)\in[0; T]\times\overline{\Theta}$ two random fields
\begin{eqnarray}
u(t, x) = Y_t^{t,x}, \ v(t, x) = U_t^{t,x}.\label{Visco}
\end{eqnarray}
\begin{theorem}\label{Th:exis-uniq}
Assume that $({\bf A1})$--$({\bf A5})$ hold. Then, the random field
$v$ is a stochastic viscosity solution of SPDIE $(\widetilde{f},
0,\widetilde{\phi})$ and hence $u$ is a stochastic viscosity
solution to SPDIE $(f, g, \phi)$.
\end{theorem}
\begin{proof}
Let us define $u(t,x)=Y^{t,x}_t$ and $v(t,x)=U^{t,x}_t$, where $Y$
and $U$ are given as above. We have
\begin{eqnarray}
u(t,\omega,x)=\eta(t,\omega,v(t,\omega,x))\;\; \mbox{and}\;\; v(t,\omega,x)=\varepsilon(t,\omega,v(t,\omega,x)).
\label{TD}
 \end{eqnarray}
Since $Y^{x,t}_s$ is
$\mathcal{F}_{t,s}^{L}\otimes\mathcal{F}_{s,T}^{B}$-measurable, it
follows that $Y^{x,t}_t$ is $\mathcal{F}_{t,T}^{B}$-measurable.
Therefore, $u(t,x)$ is $\mathcal{F}_{t,T}^{B}$- measurable and so it
is independent of $\omega'\in\Omega'$. Consequently, according to
proposition 1.7 in \cite{HY1}, we have $u\in C({\bf
F}^{B},[0,T]\times\overline{\Theta})$. Moreover, \eqref{TD} implies
that $v$ belongs to $C({\bf F}^{B},[0,T]\times\overline{\Theta})$.
We emphasis that as an ${\bf F}^B$-progressively measurable
$\omega$-wise viscosity solution is automatically a stochastic
viscosity solution (see Definition 2.3),
 it suffice to show that $v$ is an $\omega$-wise viscosity solution to SPDIE$(\widetilde{f},0,\widetilde{\phi})$.
 To do it, let us denote, for a fixed $\omega\in\Omega$,
\begin{eqnarray*}
\overline{U}^{\omega}(\omega')=U(\omega,\omega'),\;\;\overline{V}^{\omega}(\omega')=V(\omega,\omega').
\end{eqnarray*}
Then, $(\overline{U}^{\omega},\overline{V}^{\omega})$ is the unique
solution of the GBDSDELs with coefficient
$(\widetilde{f}(\omega,\cdot,\cdot,\cdot),\;\,\widetilde{\phi}(\omega,\cdot,\cdot))$,
and as it is shown by Ren and Otmani in \cite{OR}, $\bar{v}(\omega,
t, x)=\overline{U}^{\omega}_t$ is a viscosity solution to
SPDIE$(\widetilde{f}(\omega,\cdot,\cdot,\cdot),\widetilde{\phi}(\omega,\cdot,\cdot))$
with nonlinear Neumann boundary condition. By Blumenthal's $0$-$1$
law it folows that
$\displaystyle{\P'(\overline{U}_t^{\omega}(\omega')=U_t(\omega,\omega'))=
1}$. Hence we get $\bar{v}(t, x) = v(t, x)\;\P$-almost surely for
all $(t, x)\in[0,T]\Theta$. Therefore, for every $\omega$ fixed the
function $v\in C({\bf F}^B, [0,T]\times\Theta)$ is a viscosity
solution to the SPDE
$(\widetilde{f}(\omega,\cdot,\cdot,\cdot),\widetilde{\phi}(\omega,\cdot,\cdot))$.
Hence, by definition it is an $\omega$-wise viscosity solution and
hence a stochastic viscosity solution to
SPDIE$(\widetilde{f},0,\widetilde{\phi})$. The conclusion of the
theorem now follows from Theorem 3.5.
\end{proof}

\label{lastpage-01}
\end{document}